\documentclass{amsart}

\usepackage{amsmath,amssymb,mathtools,tikz,lineno}
\usetikzlibrary{arrows,automata,decorations.pathreplacing,decorations.pathmorphing,calligraphy,patterns,calc}

\newcommand{\R}{\mathbb{R}}
\newcommand{\Z}{\mathbb{Z}}
\newcommand{\N}{\mathbb{N}}

\newcommand{\sedge}[1]{\tikz[baseline=-0.5ex]{\draw (0,0) -- node[above] {$#1$} (2em,0);}} 
\newcommand{\dedge}[2]{\tikz[baseline=-0.5ex,every loop/.style={}]{\draw (0,0) edge[loop above] node[above] {$#1$} () edge [bend left=10] node[below] {$#2$} (3em,0) edge[bend right=10] (3em,0) (3em,0) edge[loop above] node[above] {$#1$} ();}}

\newenvironment{fsa}[1][auto]{\begin{tikzpicture}[->,>=stealth',
    shorten >=1pt,auto,node distance=3cm,double distance between line centers=0.45ex,
    initial text=,accepting/.style=accepting by arrow,
    every state/.style={inner sep=3pt,minimum size=6mm},
    every loop/.style={looseness=12},semithick,#1]}{\end{tikzpicture}}

\newtheorem{theorem}{Theorem}[section]
\newtheorem{lemma}[theorem]{Lemma}
\newtheorem{proposition}[theorem]{Proposition}
\newtheorem{corollary}[theorem]{Corollary}

\theoremstyle{definition}
\newtheorem{defi}[theorem]{Definition}

\title{Growth of groups with linear Schreier graphs}
\author{Laurent Bartholdi}
\author{Volodymyr Nekrashevych}
\author{Tianyi Zheng}
\date{\today}

\begin{document}

\maketitle

\section{Introduction}

We introduce here a new method of proving upper estimates of growth of finitely generated groups and constructing groups of intermediate growth using  graphs of their actions. If a group $G$ generated by a finite set $S$ acts from the right on a set $V$, then the \emph{graph of the action} $\Gamma$ is the graph with vertex set $V$ and for every generator $s\in S$ and every $x\in V$ an arrow starting in $x$, ending in $x\cdot s$, and labeled $s$. The action (and thus the group if the action is faithful) is uniquely determined by the labeled graph $\Gamma$. If we want to know the image of a point $x\in V$ under the action of an element $g=s_1s_2\ldots s_\ell$, where $s_i\in S\cup S^{-1}$, we just have to find the unique path $\gamma_{x, g}$ starting in $x$ whose consecutive edges are labeled by $s_1, s_2, \ldots, s_\ell$, where traversing an edge labeled by $s$ against its orientation corresponds to the label $s^{-1}$.

Note that $x\cdot g$ depends only on the corresponding path $\gamma_{x, g}$, so $x\cdot g$ depends only on a finite subgraph of $\Gamma$. In particular, it depends only on the isomorphism class (as a labeled rooted graph) of the ball of radius $n$ around $x$. Therefore, the element $g$ can be completely described by a finite collection of finite connected bi-rooted labeled graphs $\{(\Delta_i, x_i, y_i)\}_{i\in I}$ such that for every $x\in V$ there exists $i\in I$ and a morphism of labeled graphs $\phi_i\colon\Delta_i\to\Gamma$ with $\phi_i(x_i)=x$, and for every such morphism we have $\phi_i(y_i)=x\cdot g$. We call such a collection a \emph{portrait} of $g$.

One of the ingredients of our method is exploiting \emph{quasiperiodicity} (in form of \emph{linear repetitivity} of $\Gamma$), which together with the condition of polynomial growth of $\Gamma$ (all graphs in our examples will actually have linear growth) implies low (polynomial) \emph{complexity} of $\Gamma$. The latter conditions ensures that there are as few as possible (namely at most a polynomial function of $n$) isomorphism classes of balls of radius $n$ in the graph $\Gamma$. This way we will know that the elements of the portrait of an element $g\in G$ are drawn from a set of controlled size.

Merely quasiperiodicity and low complexity are not enough to guarantee sub-exponential growth of the group (though they may suffice to ensure other finiteness conditions, see for example~\cite{mattebon:liouville}). However, we observe that instead of the ball of radius $n$ around $x$ we may take any ball centered in $x$ that contains the path $\gamma_{x, g}$. The second ingredient of the method is ensuring that most paths $\gamma_{x, g}$ have small diameter, by showing that they have a lot of back-tracking.

Similar ideas were used in the paper~\cite{nek:burnside} to construct the first examples of simple groups of intermediate growth. They also appeared implicitly in the papers~\cite{BE:permextensions,BE:givengrowth} in the form of estimates of \emph{inverted orbits}.

We make these arguments more systematic and study the sizes of portraits of elements of the group using generating functions measuring complexity of portraits. An equivalent generating function counts ``traverses'' of specially chosen subgraphs  of $\Gamma$. Each of these subgraphs has marked vertices on their boundary called ``entrances'' and ``exits.'' A subword $s_is_{i+1}\cdots s_j$ of a group word $s_1s_2\ldots s_\ell$ is called a \emph{traverse} of the subgraph $I$ if a path corresponding to the subword starting in an entrance stays all the time inside $I$, ends in an exit, and does not include any boundary point of $I$ in-between.
The action graphs $\Gamma$ are in our paper infinite chains (with multiple edges and loops), and the special subgraphs $I$ are its sub-segments. One of the two endpoints of $I$ is marked as the entrance, the other one is the exit.

More precisely, let us fix throughout this paragraph a long word  $g=s_1s_2\ldots s_\ell$. In order for the path $\gamma_{x, g}$ to have a large diameter, it must contain traverses of a long segment $I$, since the isomorphic copies of $I$ appear quasiperiodically in $\Gamma$ (i.e., with relatively short gaps between consecutive occurrences). One can use this fact to estimate the generating function describing the portrait of $g$ by the generating function $F_g(x)$ counting the numbers of traverses of some chosen segments $I$.

If $I\hookrightarrow J$ is a fixed embedding of one interval into another, then every traverse $s_i\ldots s_j$ of $J$ contains (as a subword) a traverse of $I$. Choosing the first one, we obtain a map from the set of traverses of $J$ to the set of traverses of $I$. It is easy to see that this map is injective. Combining several maps of this form we get an oriented graph whose set of vertices is the set of traverses. The main idea of the estimates for the generating function $F_g(t)$ is to produce explicit inequalities between the in-degree and out-degrees of this graph, thereby proving that the number of traverses of segments $I$ decreases quickly as the length of $I$ increases.

Such inequalities are proved by exploiting a ``U-turn'' effect: if a subword $s_is_{i+1}\ldots s_j$ is a traverse of one of the chosen segments $I$, then for another chosen segment $J$ it is not a traverse, since the corresponding path must exit $J$ through the entrance. This will show that the injective maps, described above are substantially non-surjective, i.e., that the in-degrees of vertices of the graph are substantially smaller than the out-degrees.

We illustrate the method first by giving a much better estimate of the form $\exp(R^\alpha)$, $\alpha<1$ for the growth function $\gamma(R)$, in the case of the groups considered in~\cite{nek:burnside} (the original estimate in~\cite{nek:burnside} was $\le\exp(n/\exp(C\sqrt{\log R}))$). We show that our method gives the sharp upper estimate for the limit $\lim_{R\to\infty}\frac{\log\log\gamma(R)}{R}=0.76743\ldots$ for the growth $\gamma(R)$ of the Grigorchuk group (known from~\cite{bartholdi:upperbd,muchnpak:growth} and~\cite{erschlerzheng:grigorchukgroup}). We also give upper bounds of the form $\exp(R^\alpha)$ with explicit $\alpha$ for two more examples: the fragmentation of the golden mean dihedral group from~\cite{nek:burnside} and for a simple group containing the Grigorchuk group from~\cite{nek:simplegrowth}.

\section{Growth estimates via graphs of action}
\subsection{Portraits}
Let $G$ be a group generated by a finite symmetric set $S$. For $g\in G$ we denote by $|g|$ the minimal $\ell\in\N$ such that $g$ may be written as a product $g=s_1\cdots s_\ell$ with $s_i\in S$, and we let $\gamma_G(R)$ denote the number of group elements $g\in G$ with $|g|\le R$. We use the standard notation from theory of growth of groups: $f_1(R)\preceq f_2(R)$ means that there exists $C\ge1$ such that $f_1(R)\le f_2(CR)$ for all $R\ge 1$, and $f_1\sim f_2$ means $f_1\preceq f_2$ and $f_2\preceq f_1$.

Suppose that $G$ acts faithfully and transitively from the right on a set $V$. Let $\Gamma$ be the \emph{graph of the action}, defined as the labeled directed graph with set of vertices $V$ and set of arrows $V\times S$, where an arrow $(x, s)$ starts in $x$, ends in $x\cdot s$, and is labeled $s$. More generally, for an $S$-labelled graph, a word $w=s_1\cdots s_\ell\in S^*$ and a vertex $x$,  we introduce the following notation: $x\cdot w$, when defined, is the unique vertex reached by following from $x$ by following in turn edges labeled $s_1,\dots,s_\ell$. Note that this notation coincides with the right action of $G$ in the case of the graph of an action.

Let $w=s_1\cdots s_\ell$ be a word in $S$, and let $g$ be the corresponding element  of $G$. For a vertex $x\in\Gamma$ consider the trajectory $x, x\cdot s_1, x\cdot s_1s_2, \ldots, x\cdot s_1s_2\cdots s_\ell$, and let $\Delta_{w, x}$ be the subgraph of $\Gamma$ spanned by the corresponding labeled edges $(x\cdot s_1\cdots s_i,s_{i+1})$. For every morphism $\phi\colon\Delta_{w, x}\to\Gamma$ of labeled graphs we have $\phi(x)\cdot g=\phi(x\cdot g)$. Moreover, this is also true if we replace $\Delta_{w, x}$ by any subgraph of $\Gamma$ containing $\Delta_{w, x}$.

\begin{defi}
A \emph{portrait} of a group element $g\in G$ is a finite set $\mathcal{P}$ of finite bi-rooted labeled graphs such that the following two conditions are satisfied:
\begin{enumerate}
\item for every vertex $v\in\Gamma$ there exists a graph $(\Delta, x, y)\in\mathcal{P}$ and a morphism of rooted labeled graphs $\phi\colon(\Delta, x)\to (\Gamma, v)$;
\item if $(\Delta, x, y)\in\mathcal{P}$ and $\phi\colon(\Delta, x)\to (\Gamma, v)$ is a morphism, then $\phi(y)=\phi(x)\cdot g$.
\end{enumerate}
\end{defi}

For example, choose any expression $g=s_1\cdots s_\ell$ of $g$ as a product of generators; then the set consisting of a single graph representing the chain of edges labeled $s_1,\ldots, s_\ell$ is a portrait of $g$.

We are interested in portraits whose elements are taken from a restricted collection of graphs. For example, we may restrict ourselves to graphs of the form $(B_{x}(R), x, y)$, where $B_{x}(R)$ is the subgraph of $\Gamma$ spanned by the ball of radius $R$ centered in $x\in\Gamma$, and $y$ is a vertex of $B_{x}(R)$. 
We shall make use of the following choice.
\begin{defi}
  Let $w=s_1s_2\ldots s_\ell$ be a word in $S$, and let $v$ be a vertex of $\Gamma$. Denote by $R_w(v)$ be the maximum over $0\le i\le n$ of the distances $d(v, v\cdot s_1s_2\cdots s_i)$. Then the \emph{standard portrait} $\mathcal{P}_w$ of $w$ is defined as the collection of the (isomorphism classes of the) bi-rooted graphs $(B_v(R_w(v)), v, v\cdot w)$. In other words, it is the collection of minimal-radius balls that cover the path labeled $w$ starting at their center. It is easy to see that $\mathcal P_w$ is a portrait of the product $g=s_1s_2\cdots s_\ell\in G$. 
\end{defi}

For $g\in G$, denote by $L(g)$ the minimal size of standard portraits $\mathcal{P}_w$ among all words $w$ representing $g$; and set $L(R)=\max_{|g|\le R} L(g)$.

Let $\gamma_\Gamma(R)=\max_{v\in\Gamma}\#B_v(R)$ be the maximal size of a ball of radius $R$ in $\Gamma$, and let $\delta_\Gamma(R)$ denote the number of isomorphism classes of rooted graphs $(B_v(R), v)$. We then have the following straightforward estimate of the growth of $G$:
\begin{proposition}
\label{pr:growth1}
The growth $\gamma_G(R)$ of $G$ satisfies
\[\gamma_G(R)\le (\delta_\Gamma(R)\gamma_\Gamma(R))^{L(R)}.\]
\end{proposition}

In particular, if $\delta_\Gamma(R)$ and $\gamma_\Gamma(R)$ are bounded above by polynomials, and $L(R)$ is bounded by $R^\alpha$ for some $\alpha\in (0, 1)$, then the growth of $G$ is bounded from above by $\exp(\log R\cdot R^\alpha)$.

\begin{proof}
Let $\mathcal{P}_w$ be the standard portrait of a word $w\in S^*$, with $|w|\le R$. Then for each $(B_{x}(r), x, y)\in\mathcal{P}_w$ we have $r\le |w|$, so the number of choices for $(B_{x}(r), x)$ is at most $\delta_\Gamma(R)$: the isomorphism type of $(B_{x}(r),x)$ is determined by that of $(B_{x}(R),x)$, and $r$ is determined as the minimal radius of a ball centered at $x$ and containing the path $w$. For each $(B_{x}(r), x)$ the number of choices for $y$ is at most $\#B_{x}(r)\le\gamma_\Gamma(R)$. It follows that each element of $\mathcal{P}_w$ is drawn from a set consisting of at most $\delta_\Gamma(R)\gamma_\Gamma(R)$ elements. Since we draw at most $L(R)$ elements, it follows that the number of possible portraits $\mathcal{P}_w$ with $|w|\le R$ is at most $(\delta_\Gamma(R)\gamma_\Gamma(R))^{L(R)}$.
\end{proof}

Consider the standard portrait $\mathcal{P}_w$ of a word $w\in S^*$, and define
\[N_p(w)=\left(\sum_{(B, x, y)\in\mathcal{P}_w}(\#B)^{p-1}\right)^{1/p}.\]
Let $N_p(g)$ be the infimum of $N_p(w)$ over all words $w$ representing $g$. Note $N_1(w)=\#\mathcal{P}_w$, so $N_1(g)=L(g)$.

Let us assume that the graph $\Gamma$ satisfies the following two conditions:
\begin{enumerate}
\item there exist constants $C, d\ge1$ such that for every vertex $v\in\Gamma$ and every $R\ge 1$ we have
\[C^{-1}R^d\le\#B_v(R)\le CR^d;\]
\item the graph $\Gamma$ is linearly repetitive, namely, there exists a constant $L>1$ such that for every vertices $v, v'\in\Gamma$ and every $R>0$ there exists an isomorphic embedding of $B_{v}(R)$ into $B_{v'}(LR)$.
\end{enumerate}

\begin{lemma}\label{lem:delta}
  Under these assumptions, there is a constant $K$ such that $\delta_\Gamma(R)\le K R^d$.
\end{lemma}
\begin{proof}
  Fix $v_0\in\Gamma$. Consider a ball $B_v(R)$ in $\Gamma$; then by~(2) there exists an  embedding $\iota\colon B_v(R)\hookrightarrow B_{v_0}(L R)$, and the isomorphism type of $B_v(R)$ is uniquely determined by $\iota(v)$, which by~(1) may assume at most $C(LR)^d$ values.
\end{proof}

\begin{proposition}\label{prop:p1p2}
For all $1\le p_1<p_2$ there exists a constant $C_{p_1, p_2}$ such that, for all $g\in G$,
\[N_{p_1}(g)\le C_{p_1, p_2}N_{p_2}(g).\]
\end{proposition}

\begin{proof}
It is enough to prove the inequality $N_{p_1}(w)\le C_{p_1, p_2}N_{p_2}(w)$ for every word $w\in S^*$.

Let $V$ be a positive number, which we will choose later. Split the sum in the definition of $N_{p_1}(w)^{p_1}$ into two parts: all summands $(\#B)^{p_1-1}$ with $\#B\le V$ and the others. 

If $v\in\Gamma$ and $R$ are such that $\#B_v(R)\le V$, then $R\le (CV)^{1/d}$. Every ball $B_v(R)$ appearing in the standard portrait $\mathcal{P}_w$ is uniquely determined by $v$ and $w$, so the number of elements $(B, x, y)\in\mathcal{P}_w$ with $\#B\le V$ is at most $\delta_\Gamma((CV)^{1/d})$, which by Lemma~\ref{lem:delta} is bounded by $C_1 V$. Consequently, the sum of the first part is at most $C^2L^d V^{p_1}$.

For each summand in the second part we have $\frac{(\#B)^{p_1-1}}{(\#B)^{p_2-1}}=
(\#B)^{p_1-p_2}\le V^{p_1-p_2}$, so the sum of the second part is at most 
$V^{p_1-p_2}N_{p_2}(w)^{p_2}$. We get
\[N_{p_1}(w)^{p_1}\le C_1V^{p_1}+V^{p_1-p_2}N_{p_2}(w)^{p_2}.\]
Choose now $V=\left(\frac{p_2-p_1}{C_1(p_1+1)}\right)^{1/p_2}N_{p_2}(w)$. We then get
\begin{multline*}N_{p_1}(w)^{p_1}\le\\ C_1\left(\frac{p_2-p_1}{C_1(p_1+1)}\right)^{p_1/p_2}N_{p_2}(w)^{p_1}+\left(\frac{p_2-p_1}{C_1(p_1+1)}\right)^{(p_1-p_2)/p_2}N_{p_2}(w)^{p_1}=\\
N_{p_2}(w)^{p_1}\cdot C_1^{(p_2-p_1)/p_2}\left(\left(\frac{p_2-p_1}{p_1}\right)^{p_1/p_2}+\left(\frac{p_2-p_1}{p_1}\right)^{(p_1-p_2)/p_2}\right),
\end{multline*}
which finishes the proof with $C_{p_1, p_2}=C_1^{\frac{1}{p_1}-\frac{1}{p_2}}\frac{(p_2-p_1)^{\frac{1}{p_1}+\frac{1}{p_2}}}{p_1^{\frac{1}{p_2}}p_2^{\frac{1}{p_1}}}$.
\end{proof}

\begin{corollary}
\label{cor:growth}
Suppose that the graph of the action $\Gamma$ satisfies the conditions (1) and (2).
Denote by $N_p(R)$ be the maximum of $N_p(g)$ for all elements $g\in G$ of length at most $R$. Then for every $p\ge 1$ there exists $K>0$ such that the growth of $G$ satisfies
\[\gamma_G(R)\le R^{K N_p(R)}.\]
\end{corollary}

\begin{proof}
By Lemma~\ref{lem:delta} we get $\delta_\Gamma(R)\le C_1 R^d$. By (1) we also have $\gamma_\Gamma(R)\le C R^d$. By Proposition~\ref{prop:p1p2} we have $N_1(R)\le C_2N_p(R)$ for some constant $C_2$ not depending on $R$. The statement of the corollary follows then from Proposition~\ref{pr:growth1}.
\end{proof}

\subsection{Traverses of segments}
Suppose now that $\Gamma$ is an infinite chain, possibly with multiple edges and loops. A \emph{segment} $I$ of $\Gamma$ is a finite connected subgraph of $\Gamma$ induced by its set of vertices. The \emph{length} $|I|$ of a segment $I$ is the number of its vertices minus one. Each segment $I$ has two \emph{endpoints}. A \emph{direction} on $I$ is a choice of one endpoint as \emph{initial}. The other endpoint is called \emph{final}.

If we fix a direction (``left-to-right'') on the chain $\Gamma$, then each subsegment $I\subset\Gamma$ comes with the \emph{induced} direction, where the left endpoint is considered to be the initial one.

We assume that $\Gamma$ is linearly repetitive in the sense that isomorphic copies (such that the isomorphism preserves the induced direction on the segments) of every segment $I$ of $\Gamma$ appear with gaps between neighboring appearances of length bounded above by $C|I|$ for some fixed $C$.
Then $\Gamma$ satisfies the conditions (1) and (2) of the previous subsection.

Let us choose a sequence $I_n$ of directed segments of $\Gamma$ and assume that  there are constants $C_1, C_2>1$ such that 
\[C_1\le|I_{n+1}|/|I_n|\le C_2\]
for all $n\ge 1$. For example, we can take an arbitrary sequence of segments $I_n$ of lengths $2^n$.

Let us consider a word $s_1s_2\cdots s_\ell\in S^*$ assumed fixed throughout this section. A \emph{traverse} of $I_n$ is a subinterval $[i,j]\subseteq[1,\ell]$ such that, letting $v$ be the initial vertex of $I_n$, all the vertices $v\cdot s_i, v\cdot s_is_{i+1}, \ldots, v\cdot s_is_{i+1}\cdots s_{j-1}$ are internal vertices of $I_n$ and $v\cdot s_is_{i+1}\cdots s_j$ is the final vertex of $I_n$. By abuse of notation we think of a traverse of $I_n$ as a subword $s_i s_{i+1}\cdots s_j$, though identical subwords can appear as different traverses.

For a word $w\in S^*$, define
\begin{equation}\label{eq:traversegf}
  F_w(t)=\sum_{n=0}^\infty\#\{\text{traverses of $I_n$ for the word $w$}\}t^n.
\end{equation}
Note that this is a polynomial in $t$: the number of non-zero summands is finite.

\begin{proposition}
\label{pr:traverses}
Suppose that there are $t_0>1$ and a function $\Phi(t)$ such that $F_w(t)\le |w|\cdot \Phi(t)$ for all $w\in S^*$ and all $t<t_0$. Set $\beta=\limsup_{n\to\infty}|I_n|^{1/n}$. Then the growth of $G$ satisfies
\[\gamma_G(R)\preceq e^{R^\alpha}\]
for all $\alpha>\frac{\log\beta}{\log\beta+\log t_0}$.
\end{proposition}

\begin{proof}
  Let $(B, x, y)$ be an element of the standard portrait $\mathcal{P}_w$. Then $B$ is a segment with center $x$ and one endpoint $y$, such that the half of $B$ containing $y$ is covered by the trajectory $w$ starting in $x$. There exists $m$ such that $I_m$ contains an isomorphic copy of $B$ and $|I_m|/\#B$ is uniformly bounded. There also exists $n$ such that we can choose copies of $I_n$ inside $I_m$ with gaps between consecutive appearances of length not greater than $L|I_n|$, the radius of $B$ is greater than $(2L+1)|I_n|$, and the ratio $\#B/|I_n|$ is bounded. Then the copy of $(B, x, y)$ in $I_m$ contains a copy $I_n'$ of $I_n$ in the half covered by the trajectory starting in $x$; let us choose arbitrarily one such copy $I'_n$. This produces a traverse $\chi(B, x, y)$ of $I_n'$ \emph{corresponding} to the element $(B, x, y)$ of the portrait: if the path traced by $w$ in $B$ first enters $I_n'$ at time $i$, and first exits $I_n'$ at time $j$, this traverse is the interval $[i,j]$. We have defined a map $\chi\subset\mathcal P_w\times\{\text{traverses}\subseteq[1,\ell]\}$.

Note that if we know the traverse and the corresponding copy $I_n'$ of $I_n$ inside of $I_m$, then we know $(B, x, y)$, since we know then a vertex of $w$'s trajectory.
Since the ratio $|I_m|/|I_n|$ is uniformly bounded, the difference $m-n$ is uniformly bounded, and the chosen copies of $I_n$ are disjoint, the number of chosen copies of $I_n$ inside $I_m$ is uniformly bounded. It follows that there is a constant $C$ such the map $\chi$ is at most $C$-to-one. The map puts into correspondence to a summand $(\#B)^{p-1}$ in the definition of $N_p(w)^p$ a summand $t^n$ in the definition of $F_w(t)$.

Let $\beta_1$ be an arbitrary number greater than $\beta$. Then there exists $K>0$ such that $|I_n|\le K\beta_1^n$ for all $n\ge 1$. If we replace $t$ by $\beta_1^{p-1}$, then the map $\chi$ associates a summand $(\#B)^{p-1}$ of $N_p(w)^p$ to a summand $\beta_1^{n(p-1)}$ of $F_w(\beta_1^{p-1})$. The ratio of the summands $(\#B)^{p-1}/\beta_1^{n(p-1)}$ satisfies
\[\frac{(\#B)^{p-1}}{\beta_1^{n(p-1)}}=\left(\frac{\#B}{\beta_1^n}\right)^{p-1}\le\left(K\cdot\frac{\#B}{|I_n|}\right)^{p-1},\]
hence is bounded by a constant not depending on $n$ nor $w$. Since the map $\chi$ is at most $C$-to-one, there exists $C_1(p)$ such that
\[N_p(w)^p\le C_1(p) F_w(\beta^{p-1}).\]
Consequently, \[N_p(w)^p\le C_2(p)\cdot |w|\] for some function $C_2(p)$ and all $p$ such that $\beta_1^{p-1}<t_0$, i.e., for all $p<1+\frac{\log t_0}{\log\beta_1}$.
It follows from Corollary~\ref{cor:growth} that the growth of $G$ satisfies $\gamma_G(R)\preceq e^{R^\alpha}$ for all $\alpha>(1+\frac{\log t_0}{\log\beta})^{-1}$.
\end{proof}

\section{Groups with a purely non-Hausdorff singularity}

We improve here one of the two main theorems of~\cite{nek:burnside} (and give a shorter proof). Let us first recall some notions related to an action of a group $G$ on a topological space $\mathcal X$. For $\xi\in\mathcal X$ we denote by $G_\xi$ the stabilizer of $\xi$, and by $G_{(\xi)}$ the group of elements $g\in G$ such that $g$ acts trivially in a neighborhood of $\xi$. The point $\xi$ is \emph{regular} if $G_\xi=G_{(\xi)}$ and \emph{singular} otherwise; in all cases, $G_{(\xi)}$ is a normal subgroup of $G_\xi$, and the \emph{group of germs} of $\xi$ is the quotient $G_\xi/G_{(\xi)}$.

For $\xi\in\mathcal X$, the orbital graph of $\xi$ is the action graph $\Gamma_\xi$ on the coset space $G_\xi\backslash G$, while the \emph{graph of germs} of $\xi$ is the action graph $\widetilde\Gamma_\xi$ on the coset space $G_{(\xi)}\backslash G$. It is a Galois covering of the orbital graph with group of deck transformations isomorphic to the group of germs $G_\xi/G_{(\xi)}$.

\begin{theorem}\label{thm:nek} Let $G$ be a group generated by a finite set $S$ of involutions and acting faithfully on a Cantor set $\mathcal X$. Suppose that there exists a singular point $\xi\in\mathcal X$ such that
\begin{enumerate}
\item the orbital graphs of regular points are linearly repetitive;
\item the orbital graphs of regular points are quasi-isometric to $\R$, while the orbital graph of $\xi$ is quasi-isometric to $[0, \infty)$;
\item the group of germs of $\xi$ is finite and for every $g\in G_\xi$ the interior of the set of fixed points of $g$ accumulates on $\xi$.
\end{enumerate}
Then there exists $\alpha\in (0, 1)$ such that the growth of $G$ satisfies
\[\gamma_G(R)\preceq \exp\left(R^\alpha\right).\]
\end{theorem}

\begin{proof}
Let $\widetilde\Gamma_\xi$ be the graph of germs of the point $\xi$, and let $H$ be its group of germs. Denote by $P_h^0$ the interior of the set of fixed points of $h\in H$. It is well defined up to taking intersections with neighborhoods of $\xi$, and is always non-empty by the conditions of the theorem. Denote by $P_h^1$ the set of points moved by $h$ (also well defined up to intersection with neighborhoods of $\xi$). The sets $P_h^1$ are non-empty for every $h\in H$, by the definition of the group of germs. Since $H$ is abelian, the sets $P_h^i$ are $H$-invariant. For every choice of a map $\kappa\colon H\to\{0, 1\}$ consider the intersection $\bigcap_{h\in H\setminus\{1\}}P_h^{\kappa(h)}$. We call the obtained sets \emph{pieces}. Each piece is an open $H$-invariant set accumulating on $\xi$. Let $\{P_1, P_2, \ldots, P_d\}$ be the set of all pieces. Every piece defines a map $\pi_i\colon H\to\Z/2\Z$ by the condition that $\pi_i(h)=0$ if $\zeta\cdot h=\zeta$ for all $\zeta$ in the intersection of $P_i$ with some neighborhood of $\xi$, and $\pi_i(h)=1$ otherwise. Note that $\pi_i(h)=1$ implies that $\zeta\cdot h\ne\zeta$ for all points in the intersection of $P_i$ with a sufficiently small neighborhood of $\xi$. It follows from the fact that $H\cong(\Z/2\Z)^k$ that $\pi_i$ are well defined epimorphisms.

Let $P_i$ be a piece. Consider a sequence $\zeta_n\in P_i$ of regular points converging to $\xi$. Then for all $n$ large enough the germ $(h, \zeta_n)$ is trivial if $\pi_i(h)=0$ and is non-trivial if $\pi_i(h)=1$. It follows that the limit $\Lambda_i$ of 
the rooted graphs $(\Gamma_{\xi_n}, \xi_n)$ is the quotient of the graph of germs $\widetilde\Gamma_\xi$ by the subgroup $H_i\coloneqq\ker\pi_i$ of the group of deck transformations $H$ of the covering $\widetilde\Gamma_\xi\to\Gamma_\xi$. We get an associated sequence of Galois coverings $\widetilde\Gamma_\xi\to\Lambda_i\to\Gamma_\xi$ with groups of deck transformations $H_i$ and $H/H_i\cong\Z/2\Z$, respectively.

Choose a quasi-isometry $\phi\colon\Gamma_\xi\to [0, \infty)$. We assume that $\phi(\xi)=0$ and that all values of $\phi$ on vertices of $\Gamma_\xi$ are integers.
There exists $L$ such that $|\phi(v)-\phi(v\cdot s)|$ is less than $L$ for all vertices $v$ and all generators $s$. Let $L'$ be such that if $|\phi(v)-\phi(u)|\le L$, then distance between $v$ and $u$ is less than $L'$.

Define $I_0=\phi^{-1}([0, L_0])$, where $L_0$ is big enough (to be chosen later). Let us define sets $I_n=\phi^{-1}([0, L_n]$ inductively in the following way. Suppose that $L_n$ was defined. Consider the preimages $D_{i, n}$, $i=1, 2, \ldots, d$, of $I_n$ under the covering map $\Lambda_i\to\Gamma_\xi$. Then, by linear repetitivity, there exists $N$ such that the graph induced on $\phi^{-1}([0, N])$ contains the $L_0$-neighborhoods of copies of $D_{i, n}$ for every $i=1, 2, \ldots, d$. Moreover, we can choose $N$ such that $N\le C_1L_n$ for some constant $C_1$, using linear repetitivity and the fact that $\phi$ is a quasi-isometry. Define then $L_{n+1}=N$, and $I_{n+1}=\phi^{-1}([0, L_{n+1}])$. We will have then $|I_n|\le C_2A^n$ for some constants $C_2, A>1$.

Consider a word $w=s_1s_2\cdots s_\ell\in S^*$, remembering that $S$ is symmetric. A \emph{walk} associated with $w$ in a Schreier graph of $G$ is given by a triple $(i, j, v)$, where $1\le i\le j\le\ell$ and $v$ is a vertex of the graph. Its \emph{trajectory} is the path
\[v, v s_i, v s_i s_{i+1}, \ldots, v s_is_{i+1}\cdots s_j\]
in the Schreier graph. The vertices $v$ and $v s_is_{i+1}\cdots s_j$ are the \emph{initial} and \emph{final} vertices of the walk. A \emph{subwalk} of the walk is a walk of the form $(i', j', v')$, where $i\le i'\le j'\le j$ and $v'=v s_i\cdots s_{j'-1}$. The trajectory of a subwalk is a subpath of the trajectory of the walk.

For $n\in\N$ and $h\in H\setminus\{1\}$, denote by $\Theta_{n,h}$ the set  of walks in $\Gamma_\xi$ associated with $w$ such that the initial and final vertices are outside of $I_n$, all the other vertices belong to $I_n$, and such that if $\tilde v$ is a preimage of the initial vertex $v$ in the covering graph $\widetilde\Gamma_\xi$, then $\tilde v\cdot h$ and $\tilde v\cdot s_is_{i+1}\cdots s_j$ are at distance at most $L'$. In other words, the lift of the walk starts in a branch of $\widetilde\Gamma_\xi$ corresponding to $1\in H$ and exits in the branch corresponding to $h\in H$, see Figure~\ref{fig:proofgrowth}.

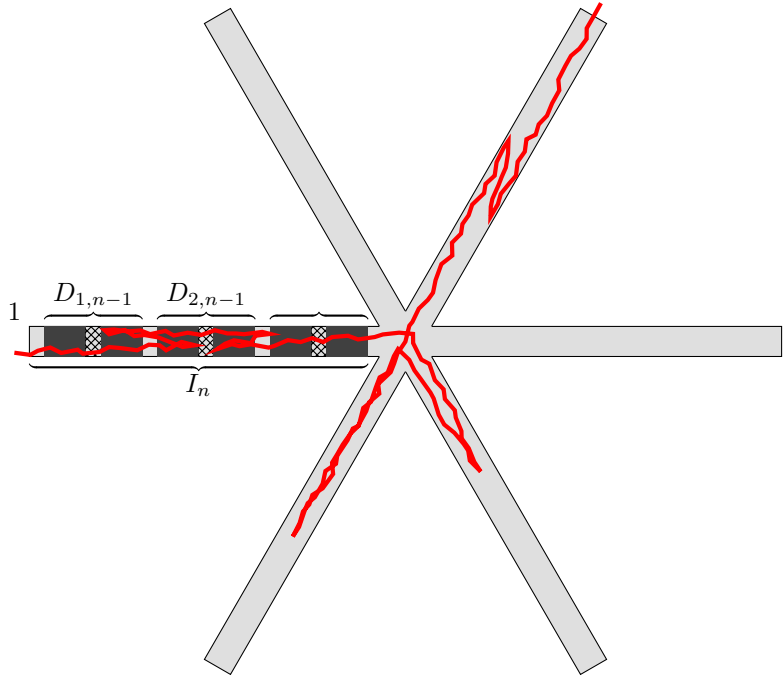
\begin{figure}
  \begin{tikzpicture}
    \foreach\i in {0,60,...,300} {
      \filldraw[fill=lightgray!50,rotate=\i] (0,-0.2) -- ++(5,0) -- ++(0,0.4) -- ++(-5,0);
    }
    \fill[lightgray!50] (0,0) circle (4mm);
    \node at (-5.2,0.4) {$1$};
    \draw [thick,decorate,decoration={calligraphic brace}] (-0.5,-0.3) -- node[below] {$I_n$} (-5,-0.3);
    \draw [thick,decorate,decoration={calligraphic brace}] (-4.8,0.3) -- node[above] {$D_{1,n-1}$} ++(1.3,0);
    \draw [thick,decorate,decoration={calligraphic brace}] (-3.3,0.3) -- node[above] {$D_{2,n-1}$} ++(1.3,0);
    \draw [thick,decorate,decoration={calligraphic brace}] (-1.8,0.3) -- ++(1.3,0);
    \foreach\x in {-4.8,-3.3,-1.8} {
      \fill[darkgray] (\x,0.2) rectangle ++(0.55,-0.4);
      \fill[pattern=crosshatch] (\x+0.55,0.2) rectangle ++(0.2,-0.4);
      \fill[darkgray] (\x+0.75,0.2) rectangle ++(0.55,-0.4);
      
    }
    \draw[decorate,decoration={random steps,segment length=4.5pt,amplitude=1.5pt},red,ultra thick] (-5.2,-0.15) -- (-2.8,-0.05) -- (-4,0.15) -- (-1.8,0.1) -- (-2.5,-0.1) -- (0.1,0.1) -- (300:2) -- (-0.1,-0.1) -- (240:3) -- (0,0) -- ($(60:3)+(150:0.15)$) -- ($(60:2)+(330:0.15)$) -- (60:5.2);
  \end{tikzpicture}
\label{fig:proofgrowth}
\caption{The graph of germs $\widetilde\Gamma_\xi$ in the proof of Theorem~\ref{thm:nek}}
\end{figure}

Note that if $x$ and $y$ are the initial and final vertices of an element of $\Theta_{n,h}$, then $\phi(x), \phi(y)\notin [0, L_n]$, but each of them is at distance at most $L$ from a point of $\phi(I_n)\subset[0, L_n]$. It follows that $|\phi(x)-\phi(y)|<L$, hence the distance between $x$ and $y$ inside $\Gamma_\xi$ is less than $L'$. For every preimage $\tilde x$ of $x$ under the covering map $\widetilde\Gamma_\xi\to\Gamma_\xi$ there exists a preimage $\tilde y$ of $y$ such that $|\tilde x-\tilde y|<L'$. Since any two different preimages of $y$ can be mapped to different points of $\Lambda_i$, if we take $L_0$ big enough, then the preimage $\tilde y$ of $y$ will be uniquely determined by the condition $|\tilde x-\tilde y|<L'$. It follows that the sets $\Theta_{n,h}$ are disjoint. Note that there is a uniform bound (not depending on $g$, $h$, or $n$) on the cardinality of the set of possible initial vertices of an element of $\Theta_{n,h}$ for every given $n$. We also write $\Theta_n=\bigcup_{h\in H\setminus\{1\}}\Theta_{n,h}$.

Let us finally write $F(t)=\sum_{n=0}^\infty \#\Theta_n t^n$ the traverse counting generating function as in~\eqref{eq:traversegf}, and its variant $F_h(t)=\sum_{n=0}^\infty \#\Theta_{n,h}t^n$ for every $h\in H$. Since the graphs $\Lambda_i$ are locally contained in every orbital graph of the action $(G, \mathcal X)$, and the action is linearly repetitive, the same arguments as in Proposition~\ref{pr:traverses}, using an inequality $|I_n|\le C_2A^n$, show that it is enough to show that $F(t)\le |w|\cdot \Phi(t)$ for all $w$ and all $t<t_0$ for some $t_0>1$.

The trajectory of every walk $\gamma\in \Theta_{n,h}$ must pass through the copy inside $I_n$ of the preimage $D_{i, n-1}$ of $I_{n-1}$ in $\Lambda_i$.
The corresponding subwalk $\gamma'$ is lifted to a walk $\gamma''$ in $\widetilde\Gamma_\xi$ belonging to $\Theta_{n-1,h'}$ for some $h'\in H$ such that $\pi_i(h')\ne 0$. The walk $\gamma''$  together with its initial vertex uniquely determine $\gamma'$, and the latter uniquely determines $\gamma$ as the unique element of $\Theta_{n,h}$ containing $\gamma'$ as a sub-walk. We get
\[\sum_{h\in H\setminus\{1\}}\#\Theta_{n,h}\le\sum_{h\notin\ker\pi_i} \#\Theta_{n-1,h},\] 
hence
\[t^{-1}(F(t)-\#\Theta_0)\le\sum_{h\notin\ker\pi_i}F_h(t)\]
for every $i=1, 2, \ldots, d$ and every $t>0$.
Since every element of $H$ belongs to at least one kernel $\ker\pi_i$, when we add the inequalities together, we get
\[dt^{-1}(F(t)-\#\Theta_0)\le (d-1)F(t),\]
hence
\[F(t)\le \#\Theta_0\cdot\frac{t^{-1}}{t^{-1}-\frac{d-1}{d}}\le |w|\cdot\frac{t^{-1}}{t^{-1}-\frac{d-1}{d}}\]
for all $t$ for which the denominator on the right-hand side is positive, i.e., for all $t\in\left(0, \frac{d}{d-1}\right)$. This finishes the proof of the theorem.
\end{proof}

\section{Examples}
We compute, for three concrete examples of groups, upper bounds of the form $\exp(n^\alpha)$ on their growth function. The first example is given as an illustration of the method, recovering the optimal bound from~\cite{bartholdi:upperbd}, while the other two give examples of virtually simple groups of intermediate word growth.

We summarize the general strategy. In all cases, there will be a group $G$ acting on the Cantor set, and the Schreier graph $\Gamma$ of a singular point $\xi$ is a half-line (with multiple edges and loops). Let $\tilde\Gamma$ be the graph of germs of the action at $\xi$. We will consider subintervals $I_n\subset\Gamma$ and the corresponding subgraphs $\tilde I_n\subset\Gamma_n$, and derive relations between the traverses of $I_n$ and $\tilde I_n$.

More generally, our tools are the following. Recall that a word $s_1\dots s_\ell$ is always under consideration, and considered fixed throughout the argument. For a bi-rooted graph $(I,x,y)$, denote by $\Theta(I,x,y)$ its set of traverses, namely of intervals $[i,j]$ such that $x s_i\cdots s_j=y$ and the path $x,x s_i,\dots,y$ remains entirely inside $I$. For a rooted graph $(\tilde I,\tilde x)$ and a graph covering $\phi\colon(\tilde I,\tilde x)\to(I,x)$, we have
\begin{equation}\label{eq:cover}
  \Theta(I,x,y)=\sum_{\tilde y\in\phi^{-1}(y)}\Theta(\tilde I,\tilde x,\tilde y);
\end{equation}
indeed every traverse of $I$ may be uniquely lifted to a traverse of $\tilde I$ that ends in some well-defined preimage of $w$. Next, for a bi-rooted graph $(I',x',y')$ with $I'\subset I$ and such that $I\setminus I'$ is separated in two connected components, one containing $\{x,x'\}$ and one containing $\{y,y'\}$, there is an embedding
\begin{equation}\label{eq:restrict}
  \iota_{I,I'}\colon\Theta(I,x,y)\hookrightarrow\Theta(I',x',y')
\end{equation}
defined as follows: $\iota([i,j])=[i',j']$ with $i'\le j$ maximal such that $x s_i\dots s_{i'}=x'$ and $j'>i'$ minimal such that $x s_i\dots s_j'=y'$; in other words, $[i',j']$ is the last traverse of $I'$ within the traverse $[i,j]$. The two relations~\eqref{eq:cover} and~\eqref{eq:restrict} will be sufficient to obtain all our upper bounds.

\subsection{The first Grigorchuk group}
Recall that the first Grigorchuk group~\cite{grigorchuk:80_en,grigorchuk:milnor_en} is a group $G=\langle a,b,c,d\rangle$ acting on $\{0,1\}^\infty$, recursively as
\begin{xalignat*}{2}
(0w)\cdot a &=1w,	& (1w)\cdot a &=0w,\\
(0w)\cdot b &=0(w\cdot a),	& (1w)\cdot b &=1(w\cdot c),\\
(0w)\cdot c &=0(w\cdot a),	& (1w)\cdot c &=1(w\cdot d),\\
(0w)\cdot d &=0w,	& (1w)\cdot d &=1(w\cdot b).
\end{xalignat*}
This, as well as the other examples, is an instance of an \emph{automaton group}. It may be presented by a finite state automaton, namely a finite directed graph with an input and output label in $\{0,1\}$ on each edge; each vertex $s$ defines a transformation of $\{0,1\}^\infty$ as follows: given $\eta\in\{0,1\}^\infty$, there will exist a unique path in the graph whose input labels read $\eta$; the output labels on the same path read a sequence $\eta\cdot s$. The group $G$ is the group generated by these transformations as $s$ ranges over the vertices of the automaton in Figure~\ref{fig:grigaut}.
\begin{figure}
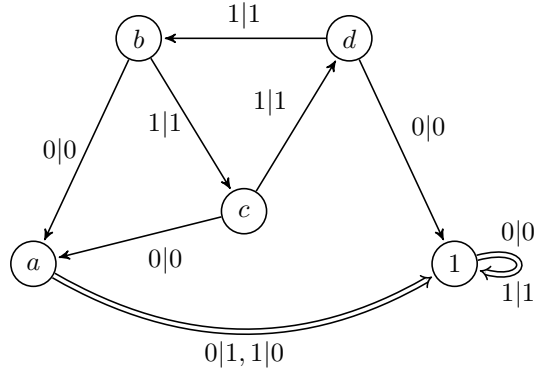

\[\begin{fsa}
  \node[state] (b) at (1.4,3) {$b$};
  \node[state] (d) at (4.2,3) {$d$};
  \node[state] (c) at (2.8,0.7) {$c$};
  \node[state] (a) at (0,0) {$a$};
  \node[state] (e) at (5.6,0) {$1$};
  \path (b) edge node[left] {$1|1$} (c) edge node[left] {$0|0$} (a)
        (c) edge node {$1|1$} (d) edge node {$0|0$} (a)
        (d) edge node[above] {$1|1$} (b) edge node {$0|0$} (e)
        (a) edge[-implies,double,bend right=30] node[below] {$0|1,1|0$} (e)
        (e) edge[-implies,double,loop right] node[above=1mm] {$0|0$} node[below=1mm] {$1|1$} (e);
      \end{fsa}
\]
\label{fig:grigaut}
\caption{The automaton generating the Grigorchuk group}
\end{figure}

The Schreier graph of $G$ may be constructed as follows: define segments $I_n$ recursively by
\[I_1=0\sedge{a}1, \qquad I_{n+1}=I_n^{-1}e_n I_n\quad\text{for }n\ge1,\]
where the edges $e_n$ are 
\begin{equation}\label{eq:grigedges}
  e_{3k}=\dedge{b}{c, d},\qquad e_{3k+1}=\dedge{d}{b, c},\qquad e_{3k+2}=\dedge{c}{b, d}.
\end{equation}
Note that it follows by induction that the segments $I_n$ are symmetric, so there is not need in writing inverses. It also follows that $I_n$ is a beginning of $I_{n+1}$, so that in the limit we get a right-infinite ray $\Gamma=I_\infty$.

The graph of germs $\widetilde\Gamma=\widetilde I_\infty$ of $1^\infty$ is obtained by taking four copies of $I_\infty$ and connecting their origins by the Cayley graph of $\{1, b, c, d\}$, see~\cite{vorob:schreiergraphs}. We label one of copies of $I_\infty$ in $\widetilde I_\infty$ by the identity element of $\{1, b, c, d\}$, and the remaining three copies by the corresponding non-trivial elements $b, c, d$.

Let $\widetilde I_n$ be the central part of the graph of germs of $1^\infty$ obtained by connecting four copies of $I_{n-1}\subset I_\infty$ at their origin.

Choose as our initial (entrance) vertex $v_0$ the boundary point of $\widetilde I_n$ in the ray $I_\infty$ labeled by the identity element of $\{1, b, c, d\}$. The three remaining boundary points correspond then to $b, c, d$ and are the exit vertices of traverses.

Let us write $\Theta_n$ for the set of traverses of $I_n$, and for $h\in\{b,c,d\}$ write $\widetilde\Theta_{n,h}$ for the set of traverses of $\widetilde I_n$ from the copy of $I_{n-1}$ labeled $1$ to that labeled $h$. Since $\widetilde I_n$ is a cover of $I_n$, we get by~\eqref{eq:cover}
\[\Theta_n=\bigsqcup_{h\in e_n}\widetilde\Theta_{n,h}.\]
Every traverse of $\widetilde I_n$ must pass through the first branch $I_{n-1}=I_{n-2}e_{n-2} I_{n-2}$ and hence through its parts $I_{n-3}e_{n-2}I_{n-3}$ and $I_{n-4}e_{n-2}I_{n-4}$. These in turn lift uniquely to traverses of $\widetilde I_{n-1}, \widetilde I_{n-2}, \widetilde I_{n-3}$ respectively, so we get by~\eqref{eq:restrict}
\[\bigsqcup_{h\in\{b,c,d\}}\widetilde\Theta_{n,h}\hookrightarrow\bigsqcup_{h\in e_{n-2}}\widetilde\Theta_{n-i,h}\text{ for }i=1,2,3,\]
see~\ref{fig:pgrigorch}. Writing $\widetilde T_{n,h}=\#\widetilde\Theta_{n,h}$ and $\widetilde T_n=\sum_{h\in\{b,c,d\}}\widetilde T_{n,h}$ and $T_n=\#\Theta_n$, we get
\[\widetilde T_{n+i}\le\sum_{h\in e_{n-2+i}}\widetilde T_n(h).\]

\begin{figure}
  \begin{tikzpicture}
    \def\w{5}
    \draw[thin] (-\w-0.5,0) node[left] {$1$} -- node[below] {$I_{n-1}$} (-0.5,0) -- (\w+0.5,0) node[right] {$b$} (-\w-0.5,-1) node[left] {$c$} -- (\w+0.5,-1) node[right] {$d$} (-0.5,0) -- ++(1,-1) -- ++(0,1) -- ++(-1,-1) -- ++(0,1);
    \node[anchor=east] at (-\w-1,-0.5) {$\widetilde I_n$};
    \foreach\i/\l in {1/5,2/3.5,3/2} {
      \draw[thin] (-3.0-0.5*\l,\i) -- +(\l,0) ++(0,-0.5) -- +(\l,0) (-3.25,\i) -- ++(0.5,-0.5) -- ++(0,0.5) -- ++(-0.5,-0.5) -- ++(0,0.5);
      \node[anchor=east] at (-6,-0.25+\i) {$\widetilde I_{n-\i}$};
    };
    \draw[very thick] (-5.5,0) -- (-0.5,0)
    (-5.5,1) -- (-0.5,1)
    (-4.75,2) -- (-3.25,2) -- ++(0,-0.5) -- (-4.75,1.5)
    (-4,3) -- (-3.25,3) -- ++(0.5,-0.5) -- (-2,2.5);
  \end{tikzpicture}
\label{fig:pgrigorch}
\caption{Traverses in the Grigorchuk group}
\end{figure}
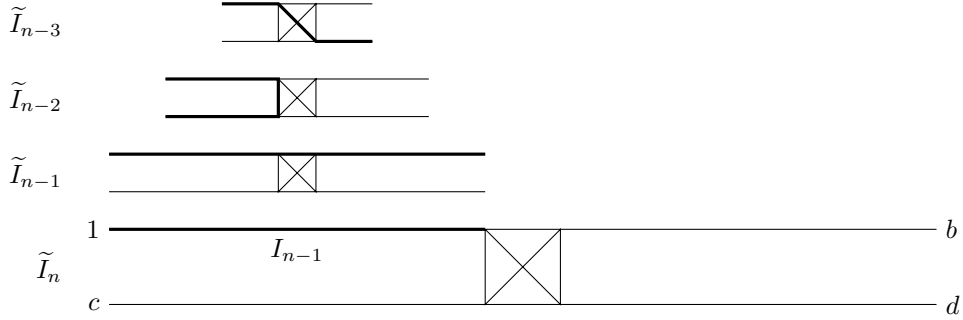

With the generating functions
\[F_h(t)=\sum_{n=0}^\infty  \widetilde T_{n,h} t^n,\qquad F(t)=\sum_{h\in\{b, c, d\}}\widetilde F_h(t),\]
and remembering that $T(n)$ is at most $|w|$, we get
\[t^{-i}\big(F(t)-C_i(t)|w|\big) \le \sum_{h\in e_{n-2+i}} F_h(t)\text{ for }i=1,2,3,\]
for some polynomials $C_i$ not depending on $w$ and all positive $t$.
Adding these inequalities together, and noting that every element $t\in \{b, c, d\}$ appears exactly twice in the labels $e_{n-1}, e_{n-2}, e_{n-3}$, we get
\[(t^{-1}+t^{-2}+t^{-3})F(t)-C(t)|w| \le 2F(t),\]
hence
\[F(t)\le |w|\cdot\frac{C(t)}{t^{-1}+t^{-2}+t^{-3}-2}\]
for some function $C(t)$ and all $t>0$ such that the denominator on the right-hand side is positive, i.e., for all $t<\eta^{-1}$, where $\eta$ is the positive root $\eta\approx 0.81054$ of $x^3+x^2+x-2$. By Proposition~\ref{pr:traverses} this implies that the growth of the Grigorchuk group is $\preceq \exp(R^\alpha)$ for every $\alpha>\frac{\log 2}{\log 2-\log\eta}\approx 0.76743$.

\subsection{An example from the golden mean rotation}
We consider next the example from~\cite[\S8]{nek:burnside}. We repeat briefly the group's definition. Consider first the ``golden mean shift'': it is the space $\mathcal X$ of sequences over the alphabet $\{0,1\}$ with no consecutive $11$. Define then bijections $a_i,b_i,c_i,d_i$ of $\mathcal X$ recursively by the formulas
\begin{xalignat*}{3}
  (00w)\cdot a_0 &= 10w, & (10w)\cdot a_0 &= 00w, & (010w)\cdot a_0 &= 010w,\\
  (00w)\cdot b_0 &= 10w, & (10w)\cdot b_0 &= 00w, & (010w)\cdot b_0 &= 010(w\cdot c_0),\\
  (00w)\cdot c_0 &= 10w, & (10w)\cdot c_0 &= 00w, & (010w)\cdot c_0 &= 010(w\cdot d_0),\\
  (00w)\cdot d_0 &= 00w, & (10w)\cdot d_0 &= 10w, & (010w)\cdot d_0 &= 010(w\cdot b_0),\\
  (0w)\cdot x_1 &= 0(w\cdot x_0), & (10w)\cdot x_1 &= 10w, & \text{ for } & x\in\{a,b,c,d\},\\
  (0w)\cdot x_2 &= 0w, & (10w)\cdot x_2 &= 10(w\cdot x_0),
\end{xalignat*}
or alternatively by the automaton in Figure~\ref{fig:goldenautomaton}, in which the action of states $x_1$ and $x_2$ should be extended by the identity where they are not defined.
\begin{figure}
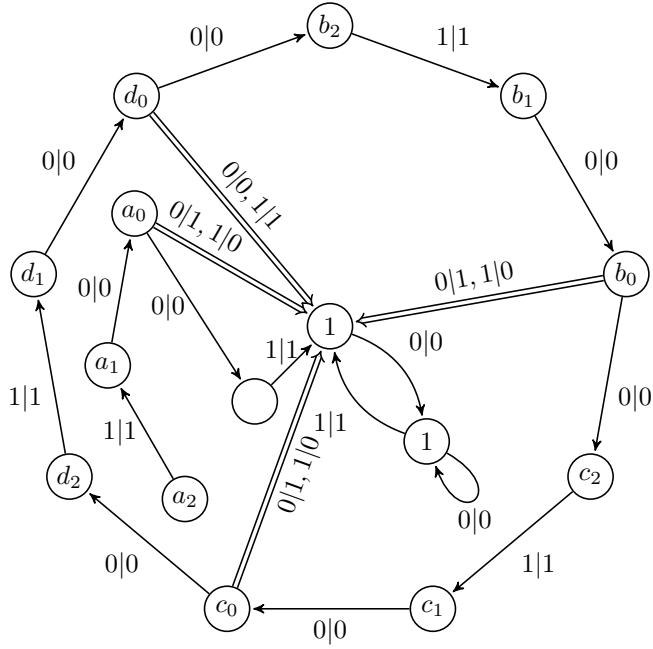
\label{fig:goldenautomaton}
\[\begin{fsa}
    \foreach\name/\angle/\lab in {b_2/90/b2,b_1/50/b1,b_0/10/b0,c_2/330/c2,c_1/290/c1,c_0/250/c0,d_2/210/d2,d_1/170/d1,d_0/130/d0} {
      \node[state,label=center:$\name$] (\lab) at (\angle:4) {};
    }
    \node[state] (e0) at (0,0) {$1$};
    \node[state] (e1) at (310:2) {$1$};
    \node[state] (x2) at (-1,-1) {};
    \node[state,label=center:$a_0$] (a0) at (150:3) {};
    \node[state,label=center:$a_1$] (a1) at (190:3) {};
    \node[state,label=center:$a_2$] (a2) at (230:3) {};
    \foreach\x/\y in {b/c,c/d,d/b} {
      \draw (\x2) edge node {$1|1$} (\x1);
      \draw (\x1) edge node {$0|0$} (\x0);
      \draw (\x0) edge node {$0|0$} (\y2);
    }
    \draw (a2) edge node[left] {$1|1$} (a1)
    (a1) edge node[left] {$0|0$} (a0)
    (a0) edge node[left] {$0|0$} (x2);
    \draw (x2) edge node[above,near start] {$1|1$} (e0)
    (a0) edge[-implies,double] node[sloped,near start] {$0|1,1|0$} (e0)
    (d0) edge[-implies,double] node[sloped] {$0|0,1|1$} (e0)
    (b0) edge[-implies,double] node[sloped] {$0|1,1|0$} (e0)
    (c0) edge[-implies,double] node[sloped,below] {$0|1,1|0$} (e0)
    (e0) edge[bend left] node {$0|0$} (e1)
    (e1) edge[out=330,in=290,loop] node[below] {$0|0$} (e1)
    edge[bend left] node {$1|1$} (e0);
  \end{fsa}
\]
\caption{The automaton generating the virtually simple torsion group of intermediate growth based on the golden mean rotation}
\end{figure}

Let us define intervals $I_n$ by the recurrence
\[I_n=I_{n-2}^{-1}e_{n-2}I_{n-1}^{-1}\text{ for }n\ge2\]
with $I_0=I_1$ the length-$0$ interval and
where (for $k>0$ and $i\in\{0, 1, 2\}$) the double edge $e_{3k+i}$ is labeled by
\[\begin{cases} b_i, c_i &\text{if }k\equiv 0\pmod{3},\\
    b_i, d_i &\text{if }k\equiv 1\pmod{3},\\
    c_i, d_i &\text{if }k\equiv 2\pmod{3}.\end{cases}
\]
(For $k=0$ there is an additional edge $a_i$). These define the labelings of the non-loop edges of segments in the Schreier graph; if there is no exiting edge labeled $s$ at a vertex, then this vertex is fixed by $s$.

There are three singular points $\xi_0=(010)^\infty$, $\xi_1=(001)^\infty$ and $\xi_2=(100)^\infty$. Their respective Schreier graphs are, for $i\in\{0,1,2\}$, the inductive limits of embeddings of $I_{3k+i}$ to the left end of $I_{3(k+1)+i}$. (Note that $I_n=I_{n-3}e_{n-4}I_{n-4}e_{n-2}I_{n-1}^{-1}$.) The graph of germs is obtained by connecting four copies of them by the Cayley graph of the four-group $\{1, b_i, c_i, d_i\}$. The beginnings of the orbit graphs of $\xi_i$ are given in Figure~\ref{fig:goldenorbits}.
\begin{figure}\label{fig:goldenorbits}
\[\begin{tikzpicture}[every node/.style={inner sep=1pt},>=stealth']
    \foreach\i/\t in {0/\xi_0,1/000,2/100,3/101,4/001,5/\scriptscriptstyle001000,6/\scriptscriptstyle101000,7/\scriptscriptstyle100000,8/\scriptscriptstyle000000,9/\scriptscriptstyle010000,10/\scriptscriptstyle010100,11/\scriptscriptstyle000100,12/\scriptscriptstyle100100} {
      \node[label={[label distance=1pt,rotate=-45]right:$\t$}] (x\i) at (\i,3) {$\cdot$};
    }
    \foreach\i/\t in {0/\xi_2,1/000,2/010,3/\scriptscriptstyle010000,4/\scriptscriptstyle000000,5/\scriptscriptstyle100000,6/\scriptscriptstyle101000,7/\scriptscriptstyle001000,8/\scriptscriptstyle001010,9/\scriptscriptstyle101010,10/\scriptscriptstyle100010,11/\scriptscriptstyle000010,12/\scriptscriptstyle010010} {
      \node[label={[label distance=1pt,rotate=-45]right:$\t$}] (y\i) at (\i,1.5) {$\cdot$};
    }
    \foreach\i/\t in {0/\xi_1,1/101,2/100,3/000,4/010,5/\scriptscriptstyle010101,6/\scriptscriptstyle000101,7/\scriptscriptstyle100101,8/\scriptscriptstyle100100,9/\scriptscriptstyle000100,10/\scriptscriptstyle010100,11/\scriptscriptstyle010000,12/\scriptscriptstyle000000} {
      \node[label={[label distance=1pt,rotate=-45]right:$\t$}] (z\i) at (\i,0) {$\cdot$};
    }
    \foreach\i/\j/\t in {x0/x1/x_1,x1/x2/x_0,x2/x3/x_2,x3/x4/x_0,x4/x5/x_1,x5/x6/x_0,x6/x7/x_2,x7/x8/x_0,x8/x9/x_1,x9/x10/x_0,x10/x11/x_1,x11/x12/x_0,
      y0/y1/x_0,y1/y2/x_1,y2/y3/x_0,y3/y4/x_1,y4/y5/x_0,y5/y6/x_2,y6/y7/x_0,y7/y8/x_1,y8/y9/x_0,y9/y10/x_2,y10/y11/x_0,y11/y12/x_1,
      z0/z1/x_0,z1/z2/x_2,z2/z3/x_0,z3/z4/x_1,z4/z5/x_0,z5/z6/x_1,z6/z7/x_0,z7/z8/x_2,z8/z9/x_0,z9/z10/x_1,z10/z11/x_0,z11/z12/x_1} {
      \draw (\i) -- node[above] {$\t$} (\j);
    }
    \draw [thick,decorate,decoration={calligraphic brace}] (0,1.8) -- node[above=3pt] {$I_2$} (1,1.8);
    \draw [thick,decorate,decoration={calligraphic brace}] (0,3.3) -- node[above=3pt] {$I_3$} (2,3.3);
    \draw [thick,decorate,decoration={calligraphic brace}] (4,-0.6) -- node[below=3pt] {$I_4$} (0,-0.6);
    \draw [thick,decorate,decoration={calligraphic brace}] (7,0.9) -- node[below=3pt] {$I_5$} (0,0.9);
    \draw [thick,decorate,decoration={calligraphic brace}] (0,3.8) -- node[above=3pt] {$I_6$} (12,3.8);

  \end{tikzpicture}
\]
\caption{The beginnings of the orbits of the critical rays $\xi_0,\xi_1,\xi_2$ under the fragmentation of the golden mean rotation, where a horizontal edge $x_i$ corresponds to a generator among $a_i,b_i,c_i,d_i$. On row $j$ we abbreviate `$w$' for the sequence `$w\xi_j$'.}
\end{figure}
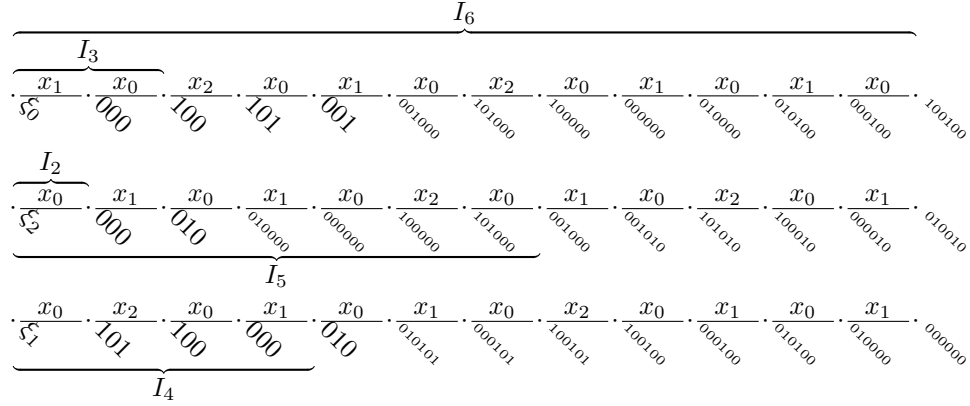

The situation is similar to that of the Grigorchuk group, except that we must consider three singular points instead of one.
Denote by $T_{3k+i,b_i}$, $T_{3k+i,c_i}$, $T_{3k+1,d_i}$ the numbers of traverses of the central part of the graph of germs formed by the four copies of $I_{3k+i}$ starting in the branch corresponding to the identity and ending in the branch of corresponding to $b_i, c_i, d_i$, respectively.  Note
\[I_n=(I_{n-2}^{-1}e_{n-2}I_{n-2})e_{n-3}I_{n-3}\]
so every traverse of $I_n^{-1}e_nI_n$ must pass through $I_{n-2}^{-1}e_{n-2}I_{n-2}$.

It follows that a traverse counted in $T_{3k+i,b_i}$ restricts to traverses counted in $\sum_{h\in e_{3k+i-2}}T_{3k+i-2,h}$, and similarly in $\sum_{h\in e_{3k+i-2}} T_{3k+i-5,h}$ and in $\sum_{h\in e_{3k+i-2}} T_{3k+i-8,h}$. Defining as before the generating functions $F_i(t)=\sum_{k=0}^\infty\sum_{h\in e_i}T_{3k+i,g_i}t^{3k+i}$, we obtain
\[(t^{-8}+t^{-5}+t^{-2})(F_i(t)-C_i(t)|w|)\le 2 F_{i-2}(t).\]
Summing over $i=0,1,2$, we get
\[(t^{-8}+t^{-5}+t^{-2})(F(t)-C(t)|w|)\le 2F(t),\]
which implies the growth estimate $\gamma(R)\preceq\exp(R^{0.9181})$.

\subsection{A simple group containing the Grigorchuk group}
The following example is studied in~\cite{nek:simplegrowth}, where it is shown that it is a virtually simple torsion group of intermediate growth containing the first Grigorchuk group.

Consider the space $\mathcal X$ of sequences over the alphabet $\{0_0, 0_1, 1\}$ such that every two-letter subword is in the set $\{0_00_1, 0_10_0, 0_11, 10_0, 10_1, 11\}$; so the $\{0_0,0_1\}$ letters alternate on every block of $0$'s, ending with $0_1$.
The map $P\colon\mathcal X\to\{0, 1\}^\omega$ erasing the indices of $0_0$ and $0_1$ is a continuous surjection such that $P^{-1}(w)$ is a singleton except when $w$ is eventually equal to the constant $0$ sequence. 

The action of the Grigorchuk group on $\{0, 1\}^\omega$ naturally lifts by $P$ to an action on $\mathcal X$. Namely, we have
\[(0_xw)\cdot a=1w,\qquad (10_xw)\cdot a=0_{1-x}0_xw,\qquad (11w)\cdot a=0_11w\]
and
\begin{xalignat*}{2}
  (0_{1-x}0_xw)\cdot b&=0_1(0_xw\cdot a),&(0_110_xw)\cdot b&=0_x(10_xw\cdot a),\\
  (0_111w)\cdot b&=0_0(11w\cdot a),&(1w)\cdot b&=1(w\cdot c),\\
  (0_{1-x}0_xw)\cdot c&=0_1(0_xw\cdot a),&(0_110_xw)\cdot c&=0_x(10_xw\cdot a),\\
  (0_111w)\cdot c&=0_0(11w\cdot a),&(1w)\cdot c&=1(w\cdot d),\\
  (0_xw)\cdot d&=0_xw,&(1w)\cdot d&=1(w\cdot b).
\end{xalignat*}

Let us decompose $a$ into a product of two disjoint involutions:
\[10_1w\stackrel{a_0}{\longleftrightarrow}0_00_1w,\qquad 1xw\stackrel{a_1}{\longleftrightarrow}0_1xw\text{ if }x\in\{0_0,1\}.\]
Denote by $\widehat G$ the group generated by $\{a_0,a_1,b,c,d\}$. It obviously contains the Grigorchuk group as the subgroup $\langle a_0a_1,b,c,d\rangle$. In fact, the orbital graphs of the Grigorchuk group are obtained from the orbital graphs of $\widehat G$ by replacing each label $0_x$ by $0$. It is given by the automaton in Figure~\ref{fig:simplegrigaut}.
\begin{figure}
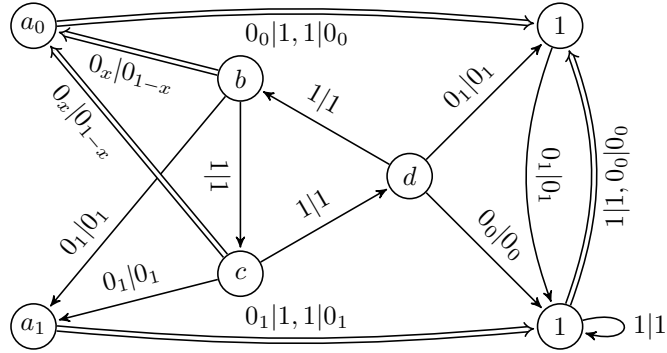

\[\begin{fsa}[every node/.style={sloped}]
  \node[state] (b) at (120:1.5) {$b$};
  \node[state] (d) at (0:1.5) {$d$};
  \node[state] (c) at (240:1.5) {$c$};
  \node[state,label=center:$a_0$] (a0) at (-3.5,2) {};
  \node[state,label=center:$a_1$] (a1) at (-3.5,-2) {};
  \node[state] (e0) at (3.5,-2) {$1$};
  \node[state] (e1) at (3.5,2) {$1$};
  \path (b) edge node[below] {$1|1$} (c) edge[-implies,double] node[below] {$0_x|0_{1-x}$} (a0) edge node[pos=0.7] {$0_1|0_1$} (a1)
        (c) edge node {$1|1$} (d) edge[-implies,double] node[below,pos=0.7] {$0_x|0_{1-x}$} (a0) edge node[above] {$0_1|0_1$} (a1)
        (d) edge node[above,pos=0.6] {$1|1$} (b) edge node {$0_0|0_0$} (e0) edge node {$0_1|0_1$} (e1)
        (a0) edge[-implies,double,bend left=5] node[below] {$0_0|1,1|0_0$} (e1)
        (a1) edge[-implies,double,bend right=5] node[above] {$0_1|1,1|0_1$} (e0)
        (e0) edge[loop right] node[sloped=false] {$1|1$} () edge[-implies,double,bend right=20] node[below] {$1|1,0_0|0_0$} (e1)
        (e1) edge[bend right=20] node {$0_1|0_1$} (e0);
      \end{fsa}
\]
\label{fig:simplegrigaut}
\caption{The automaton generating a virtually simple torsion group containing the Grigorchuk group}
\end{figure}

Segments of the orbital graphs of $\widehat G$ are constructed using the following rules:
\begin{xalignat*}{2}
  I_1 &= 0_0\sedge{a_0}1, & J_1 &= 0_1\sedge{a_1}1,\\
  I_{n+1} &= J_n e_n J_n^{-1}, & J_{n+1} &= J_n e_n I_n^{-1}\text{ for all }n\ge1,
\end{xalignat*}
where $e_n$ are the same as for the Grigorchuk group, see~\eqref{eq:grigedges}, and again loops are not indicated. Note that $I_n$ is symmetric.

The orbital graph of $1^\infty$ is isomorphic to the right-infinite ray equal to the natural direct limit of the segments $J_n$ with added loops at its origin. We will denote the direct limit by $J_\infty$. The graph of germs of $1^\infty$ is also obtained by taking four copies of $J_\infty$ and connecting them by the Cayley graph of $\{1, b, c, d\}$.

We repeat the arguments from the estimate of the growth of the Grigorchuk group, where we count the traverses of the central part $\widetilde J_n$ of the graph of germs of $1^\infty$. We have
\[J_n=J_{n-1}e_{n-1}I_{n-1}=J_{n-1}e_{n-1}J_{n-2}e_{n-2}J_{n-2}^{-1},\]
so every traverse of $\widetilde J_n$ passes through segments $I_{n-1}=J_{n-2}e_{n-2}J_{n-2}^{-1}$, $I_{n-2}=J_{n-3}e_{n-3}J_{n-3}^{-1}\subset J_{n-1}$, and $I_{n-3}=J_{n-4}e_{n-4}J_{n-4}^{-1}$.

It follows that the growth of $\widehat G$ is $\preceq\exp(R^\alpha)$ for every $\alpha>\frac{\log 2}{\log 2-\log\eta}\approx 0.83473$, where $\eta\approx 0.87176$ is the root of $x^4+x^3+x^2-2$.

\bibliographystyle{amsalpha}
\bibliography{nekrash,mymath}
\end{document}